\documentclass[a4paper,12pt]{article}
\usepackage{amsmath,amsthm,amssymb,amsfonts} 
\usepackage{bbm,bm} 
\usepackage{latexsym} 
\usepackage{mathrsfs} 
\usepackage{dsfont}
\usepackage{threeparttable} 
\usepackage{tabularx} 
\usepackage{booktabs} 
\usepackage{multirow} 
\usepackage{pifont}
\usepackage{graphicx,subfigure} 
\usepackage{color} 
\usepackage{indentfirst} 
\usepackage{geometry}
\renewcommand{\baselinestretch}{1.5} 
\usepackage{caption}
\usepackage{lineno} 
\usepackage{microtype}
\usepackage[numbers,sort&compress]{natbib} 
\usepackage[colorlinks,
            linkcolor=blue,
            citecolor=red
            ]{hyperref} 
\usepackage{epstopdf} 

\newtheorem{thm}{Theorem}[section]
\newtheorem{prop}[thm]{Proposition}

\newtheorem{lem}[thm]{Lemma}
\newtheorem{cor}[thm]{Corollary}

\begin{document}


\setlength{\baselineskip}{20pt}
\begin{center}

{\Large \bf Nice vertices in cubic graphs$^{\text{\ding{73}}\ddagger}$}
\vspace{4mm}

{Wuxian Chen$^{\rm 1}$, Fuliang Lu$^{\rm 2}$, Heping Zhang$^{\rm 1 \dagger}$}

\vspace{4mm}

\footnotesize{$^{\rm 1}$School of Mathematics and Statistics, Lanzhou University,
Lanzhou, 730000, PR China}

\footnotesize{$^{\rm 2}$School of Mathematics and Statistics, Minnan Normal University,
Zhangzhou, 363000, PR China}

\renewcommand\thefootnote{}
\renewcommand{\baselinestretch}{1.2}
\footnote{$^{\text{\ding{73}}}$ This work is supported by NSFC\,(Grant No. 12271229 and 12271235) and NSF of Fujian Province (Grant No. 2021J06029).}

\renewcommand{\baselinestretch}{1.2}
\footnote{$^{\dagger}$ Corresponding author. E-mail addresses:
chenwx21@lzu.edu.cn (W. Chen), flianglu@163.com (F. Lu) and zhanghp@lzu.edu.cn (H. Zhang).}

\renewcommand{\baselinestretch}{1.2}
\footnote{$^{\ddagger}$
Published in [{\em Discrete Math}. 348 (2025) 114553].
This version fixes an error in Lemma \ref{p-3-edge-nice}, and thus corrects the proofs of Lemmas \ref{nice-number} and \ref{6-nice}. See the appendix for details.}
\end{center}

\noindent {\bf Abstract}:
 A subgraph $G'$ of a graph $G$ is \emph{nice} if $G-V(G')$ has a perfect matching. Nice subgraphs play a vital role in the theory of ear decomposition and matching minors of matching covered graphs. A vertex $u$ of a cubic graph is \emph{nice} if $u$ and its neighbors induce a nice subgraph. D. Kr\'{a}l et al. (2010) \cite{KSS} showed that each vertex of a cubic brick is nice. It is natural to ask how many nice vertices a matching covered cubic graph has. In this paper, using some basic results of matching covered graphs, we prove that if a non-bipartite cubic graph $G$ is 2-connected, then $G$ has at least 4 nice vertices; if $G$ is 3-connected and $G\neq K_4$, then $G$ has at least 6 nice vertices. We also determine all the corresponding extremal graphs.

For a cubic bipartite graph $G$ with bipartition $(A,B)$, a pair of vertices $a\in A$ and $b\in B$ is called a \emph{nice pair} if $a$ and $b$ together with their neighbors induce a nice subgraph.
We show that a connected cubic bipartite graph $G$ is a brace if and only if each pair of vertices in distinct color classes is a nice pair. In general, we prove that $G$ has at least 9 nice pairs of vertices and $K_{3,3}$ is the only extremal graph.

\vspace{2mm}

\vspace{2mm}


{\setcounter{section}{0}
\section{\normalsize Introduction}\setcounter{equation}{0}
Graphs considered in this paper are connected, finite and simple unless specified.
A graph is \emph{cubic} if each vertex has degree 3.
Cubic graphs have been studied extensively since various open problems and conjectures can be reduced to cubic graphs.
A \emph{matching} in a graph $G$ is a set of independent edges.
A \emph{perfect matching} of $G$ is a matching covering all vertices of $G$.
In 1891, Petersen \cite{P91} showed a classical result that every 2-connected cubic graph has a perfect matching.
Sch\"{o}nberger \cite{S34} established a stronger result that each edge of a 2-connected cubic graph lies in a perfect matching.
This implies that every 2-connected cubic graph has at least three perfect matchings.
Since then, extensive research has been devoted to the enumeration of perfect matchings of cubic graphs (see literature \cite{KSS,EKK,HK,KKM}).
Meanwhile, many structural properties of cubic graphs, especially fullerene graphs, have been obtained \cite{Do02,Do03,LZ18,NDLL20}.

Let $G$ be a graph with vertex set $V(G)$ and edge set $E(G)$.
For a vertex subset $S$ of $G$, denote by $G-S$ the subgraph of $G$ obtained by deleting all the vertices in $S$ and their incident edges.
A subgraph $G'$ of $G$ is \emph{nice} if $G-V(G')$ has a perfect matching.
Nice subgraphs are closely related to the theory of ear decomposition and matching minors of matching covered graphs (cf. \cite{LP09,LM24,YL09}), which are also referred to as conformal subgraphs \cite{CLM05}, central subgraphs \cite{RST99} and well-fitted subgraphs \cite{M04}.
A connected graph with at least two vertices is \emph{matching covered} if each edge is contained in a perfect matching.
So every 2-connected cubic graph is matching covered. Note that every matching covered graph on at least four vertices is 2-connected.
Thus we have the following result.

\begin{lem}\label{matching-covered}
A cubic graph is 2-connected if and only if it is matching covered.
\end{lem}

Lov\'{a}sz and Plummer \cite{LP09} showed that every matching covered graph on at least 6 vertices has a nice even cycle of length at least 6.
Lov\'{a}sz \cite{L83} proved that every non-bipartite matching covered graph has an even subdivision of $K_4$ or $\overline{C_6}$ (the triangle prism) as a nice subgraph, where $K_i$ is a complete graph on $i$ vertices.
In chemistry, nice subgraphs of a molecular graph $G$ represent addition patterns on $G$ such that the rest of $G$ still has a resonant structure.
It has been proved that a fullerene graph (planar cubic graph with only pentagons and hexagons as faces) has some nice subgraphs, such as $K_2,K_{1,3},P_4$ in \cite{Do02}, and $C_6$ in \cite{YQZ09} and $C_8$ in \cite{LZ18}, etc., where $P_i$ (resp. $C_i$) denotes a path (resp. cycle) on $i$ vertices and $K_{m,n}$ denotes a complete bipartite graph with bipartite sets having $m$ and $n$ vertices.

A \emph{spanning subgraph} of a graph $G$ is a subgraph containing all vertices of $G$.
Kr\'{a}l et al. introduced the concept of trimatched vertices in cubic graphs in \cite{KSS}, where they proved that every 2-connected cubic graph on $n$ vertices has at least $\frac{n}{2}$ perfect matchings.
A vertex $u$ of a cubic graph $G$ is \emph{trimatched} if there exists a spanning subgraph $G'$ of $G$ such that the degree of $u$ is 3 in $G'$ and the degrees of the other vertices are 1 in $G'$.
That is, the subgraph $K_{1,3}$, which is induced by $u$ and its neighbors of $G'$, is nice in $G$.
Note that each $K_{1,3}$ subgraph of a cubic graph is nice if and only if the subgraph induced by $V(K_{1,3})$ is nice.
From this point of view, if a vertex $u$ and its neighbors of a cubic graph induce a nice subgraph, then we call $u$ a \emph{nice vertex}.
Recall that a graph $G$ with at least one edge is \emph{bicritical} if the removal of any two distinct vertices of $G$ yields a graph with a perfect matching.
Clearly, a bicritical graph is matching covered. A 3-connected bicritical graph is called a \emph{brick}.
Kr\'{a}l et. al. \cite{KSS} showed that each vertex of a cubic brick is nice. In fact, their proof also implies the following more general result.

\begin{lem}\label{bicritical}
Each vertex of a cubic bicritical graph is nice.
\end{lem}

A connected graph with at least 6 vertices and a perfect matching is \emph{2-extendable} if any two disjoint edges is contained in a perfect matching.
A 2-extendable bipartite graph is called a \emph{brace}.
Using Lemma \ref{bicritical}, Kr\'{a}l et. al. showed that every nontrivial brick and brace decomposition of a 2-connected cubic graph $G$ contains a brace, which provides an ideal to prove the main theorem by using induction on the number of braces in nontrivial brick and brace decomposition of $G$.
In this paper, we are interested in considering how many nice vertices a 2-connected cubic graph has.
By using some basic results of matching covered graphs, we prove that if a non-bipartite cubic graph $G$ is 2-connected, then $G$ has at least 4 nice vertices; if $G$ is 3-connected and $G\neq K_4$, then $G$ has at least 6 nice vertices. We also determine all the corresponding extremal graphs (see Theorem \ref{main-1}).

Obviously, a cubic bipartite graph has no nice vertices. But we can  consider a variant.
Let $G(A,B)$ be a cubic bipartite graph with bipartition $(A,B)$.
A pair of vertices $a\in A$ and $b\in B$ of $G(A,B)$ is a \emph{nice pair} if $a$ and $b$ together with their neighbors induce a nice subgraph.
Further, a pair of vertex subsets $A'\subseteq A$ and $B'\subseteq B$ of $G(A,B)$ is a \emph{nice pair set} if each pair of vertices $a\in A'$ and $b\in B'$ is nice.
We show that a connected cubic bipartite graph $G(A,B)$ is a brace if and only if each pair of vertices $a\in A$ and $b\in B$ is nice.
Then we prove that $G(A,B)$ has a nice pair set $A'\subseteq A$ and $B'\subseteq B$ with $|A'|\geq 3$ and $|B'|\geq 3$ (see Theorem \ref{main-2}).
As a corollary, we have that $G(A,B)$ has at least 9 nice pairs of vertices and $K_{3,3}$ is the only graph attaining the lower bound (see Corollary \ref{cor}).

In Section 2, we present some notations and basic useful results of matching covered graphs.
In Section 3, we prove Theorem \ref{main-1} and determine all the extremal graphs by the splicing operation.
In Section 4, we prove Theorem \ref{main-2} and Corollary \ref{cor} and determine the corresponding extremal graphs. In the last section, we summarize our results and list some problems possibly for future research.

\section{\normalsize Preliminaries}
In the section, we give some definitions and preliminary results.
For a vertex $u$ of a graph $G$, a \emph{neighbor} of $u$ is a vertex adjacent to $u$.
Denote by $N_{G}(u)$ the set of neighbors of $u$ and $N_{G}[u]$ the closed neighborhood of $u$ in $G$, i.e., $N_{G}[u]:=N_{G}(u)\cup \{u\}$.
The \emph{degree} of $u$ in $G$, denoted by $d_G(u)$, is $|N_{G}(u)|$. An \emph{isolated vertex} of $G$ is a vertex with degree 0.
For a vertex subset $S$ of $G$, let $N_{G}(S)=\cup_{x\in S}N_G(x)$ and $G[S]$ be the subgraph of $G$ induced by $S$.
A component $K$ of $G-S$ is \emph{odd} if $|V(K)|$ is odd.
Denote by $o(G-S)$ the number of odd components of $G-S$. The following is the well-known Tutte's Theorem.

\begin{thm}[\cite{Tutte47}]\label{Tutte}
A graph $G$ has a perfect matching if and only if $o(G-S)\leq|S|$ for any $S\subseteq V(G)$.
\end{thm}

Let $G$ be a graph with a perfect matching. A \emph{barrier} of $G$ is a  set $S\subseteq V(G)$ such that $o(G-S)=|S|$.
Moreover, a barrier $S$ of $G$ is \emph{nontrivial} if $|S|\geq2$. We have the following basic property of matching covered graphs.

\begin{lem}[\cite{CLM99}]\label{Tutte-cor}
Let $G$ be a matching covered graph. Then the following statements hold.\\
{\rm (i)} $G$ is free of nontrivial barriers if and only if $G$ is bicritical,\\
{\rm (ii)} For any barrier $S$ of $G$, $G-S$ has no even components,\\
{\rm (iii)} For any barrier $S$ of $G$, $S$ is an independent set.
\end{lem}

Using Tutte's Theorem, we give a necessary and sufficient condition under which a vertex of a 2-connected cubic graph is not nice as follows.
\begin{thm}\label{criteria}
Let $G$ be a 2-connected cubic graph and $u$ be a vertex of $G$.
Then $u$ is not nice in $G$ if and only if there exists a barrier $S$ of $G$ such that $u$ is an isolated vertex of $G-S$.
\end{thm}
\begin{proof}
Sufficiency. Since $u$ is an isolated vertex of $G-S$, $N_{G}(u)\subseteq S$. Set $S'=S\setminus N_{G}(u)$.
Then $o(G-N_{G}[u]-S')=o(G-(S\cup\{u\}))=o(G-S)-1$. Since $S$ is a barrier and $d_G(u)=3$, $o(G-S)-1=|S|-1>|S|-3=|S'|$. So $o(G-N_{G}[u]-S')>|S'|$.
By Tutte's Theorem, $G-N_{G}[u]$ has no perfect matching. So $u$ is not nice in $G$.

Necessity. Since $u$ is not nice, $G-N_{G}[u]$ has no perfect matching.
By Tutte's Theorem, there exists a set $S^*\subseteq V(G-N_{G}[u])$ such that $o(G-N_{G}[u]-S^*)> |S^*|$.
Then $o(G-N_{G}[u]-S^*)\geq |S^*|+2$ by parity.
Set $N_{G}(u)=\{u_1,u_2,u_3\}$. Since $G$ is matching covered by Lemma \ref{matching-covered}, $G-\{u,u_1\}$ has a perfect matching.
By Tutte's Theorem again, $o(G-\{u,u_1\}-(S^*\cup\{u_2,u_3\}))\leq |S^*\cup\{u_2,u_3\}|=|S^*|+2$.
So $o(G-N_{G}[u]-S^*)=|S^*|+2$ and $o(G-N_{G}(u)-S^*)=o(G-N_{G}[u]-S^*)+1=|S^*|+3=|S^*\cup N_{G}(u)|$.
Thus $S:=N_{G}(u)\cup S^*$ is a barrier of $G$ such that $u$ is an isolated vertex of $G-S$.
\end{proof}

Let $X$ and $Y$ be two vertex subsets of a graph $G$ and $E_G(X,Y)$ (or simply $E(X,Y)$) the set of all edges of $G$ with one end in $X$ and the other in $Y$.
For $\emptyset\neq X\subset V(G)$, the set $E(X,\overline{X})$ is called an \emph{edge cut} of $G$, and is denoted by $\partial_G(X)$ (or simply by $\partial(X)$), where $\overline{X}:=V(G)\setminus X$.
For a vertex $u$ of $G$, we denote $\partial{(\{u\})}$ by $\partial{(u)}$ for brevity.
An edge cut $\partial(X)$ is \emph{trivial} if either $|X|=1$ or $|\overline{X}|=1$. A \emph{k-cut} is an edge cut with $k$ edges.
For an edge cut $\partial(X)$ of $G$, denote by $G/(X\rightarrow x)$ (or simply $G/X$) the graph (possible with multiple edges) obtained from $G$ by contracting $X$ into a single vertex $x$.
The two graphs $G/X$ and $G/\overline{X}$ are called the \emph{$\partial(X)$-contractions} of $G$.
We will use the following two facts related to edge cuts.

\begin{lem}[\cite{BM}, Exercise 9.3.8]\label{G/X}
If $\partial{(X)}$ is a $k$-cut of a $k$-edge-connected graph, then the two $\partial(X)$-contractions are $k$-edge-connected.
\end{lem}

\begin{lem}[\cite{NDLL20}]\label{cut-matching}
In a 3-connected graph, every nontrivial 3-cut is a matching.
\end{lem}

Let $G$ be a matching covered graph. An edge cut $C$ of $G$ is \emph{tight} if $|C\cap M|=1$ for every perfect matching $M$ of $G$.
It is clear that each trivial edge cut of $G$ is tight and the two $C$-contractions of $G$ are matching covered if $C$ is a tight cut of $G$.

\begin{lem}[\cite{NDLL20}]\label{3-cut}
Each tight cut of a 2-connected cubic graph is a 3-cut.
\end{lem}

\begin{figure}
\centering
\includegraphics{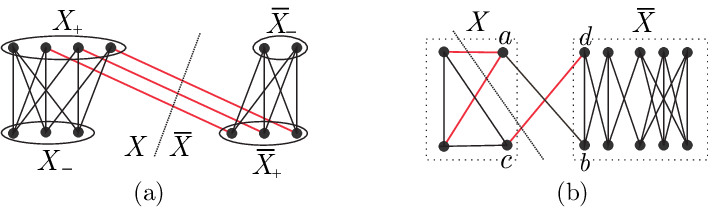}
\caption{\label{tight-cut} (a)\ A tight cut $\partial{(X)}$ in a bipartite matching covered graph, (b)\ The tight cut $\partial{(X\setminus\{a\})}$ in a non-bipartite cubic graph with a 2-cut $\partial{(X)}$.}
\end{figure}

For a vertex subset $X$ of a bipartite graph $G(A,B)$, if $|X|$ is odd, then clearly one of the two sets $A\cap X$ and $B\cap X$ is larger than the other; we denote the larger and smaller sets by $X_{+}$ and $X_{-}$, respectively, and $\overline{X}_{+}$ and $\overline{X}_{-}$ are defined analogously for the vertex subset $\overline{X}$. The following is a characterization of tight cuts in bipartite matching covered graphs (see Fig. \ref{tight-cut}(a) for an example).

\begin{lem}[\cite{L87}]\label{tight-cut-bipartite}
An edge cut $\partial{(X)}$ of a bipartite matching covered graph $G(A,B)$ is tight if and only if: {\rm (i)}\ $|X|$ is odd and $|X_{+}|=|X_{-}|+1$, and {\rm (ii)}\ $E(X_{-},\overline{X}_{-})=\emptyset$.
\end{lem}

Lov\'{a}sz \cite{L87} showed the following important theorem that plays an essential role in our proof.
\begin{thm}\label{brick-brace}
A matching covered graph has no nontrivial tight cuts if and only if it is either a brick (which is non-bipartite) or a brace.
\end{thm}

\begin{lem}[\cite{P86}]\label{brace}
A connected bipartite graph $G(A,B)$ with at least 6 vertices is a brace if and only if $G-\{a_1,a_2,b_1,b_2\}$ has a perfect matching for any four distinct vertices $a_1,a_2\in A$ and $b_1,b_2\in B$.
\end{lem}

\begin{lem}[\cite{LK20}]\label{non-brace}
Let $G$ be a bipartite matching covered graph that is not a brace.
Then $G$ has a nontrivial tight cut $C$ such that one of the two $C$-contractions is a brace.
\end{lem}

In the sequel, we show some results helpful to our proofs.

\begin{lem}\label{p-3}
Let $G$ be a 3-connected cubic graph with a nontrivial barrier $S$. Then the following statements hold.\\
{\rm (i)} For any nontrivial component $K$ of $G-S$, $\partial{(V(K))}$ is a tight 3-cut and is a matching,\\
{\rm (ii)} For any nontrivial component $K$ of $G-S$, each $\partial{(V(K))}$-contraction is a 3-connected simple cubic graph,\\
{\rm (iii)} If $G$ is non-bipartite, then $G-S$ has a nontrivial non-bipartite component.
\end{lem}
\begin{proof}
(i)\ Since $G$ is matching covered (Lemma \ref{matching-covered}), $K$ is odd by Lemma \ref{Tutte-cor}(ii).
For any perfect matching $M$ of $G$, $|\partial{(V(K))}\cap M|=1$ as $o(G-S)=|S|$. So $\partial{(V(K))}$ is a tight cut.
By Lemma \ref{3-cut}, $\partial{(V(K))}$ is a 3-cut.
Further, as $\partial{(V(K))}$ is nontrivial, $\partial{(V(K))}$ is a matching by Lemma \ref{cut-matching}, and so (i) holds.

(ii) Let $G'$ be a $\partial{(V(K))}$-contraction.
As $G$ is cubic, $G'$ is also a simple cubic graph by (i).
Since $|\partial{(V(K))}|=3$ by (i) and the fact that the edge connectivity of a cubic graph is equal to its vertex connectivity, $G'$ is 3(-edge)-connected by Lemma \ref{G/X}, and so (ii) holds.

(iii) Otherwise, suppose that all nontrivial components of $G-S$ are bipartite graphs $H_i(A_i,B_i)$, $1\leq i\leq t$, with $|A_i|\geq|B_i|$, where $1\leq t\leq|S|$.
Set $l_i=|E(A_i,S)|$, $1\leq i\leq t$.
By (i), $|\partial{(V(H_i))}|=3$ and so $0\leq l_i\leq 3$.
Since $G$ is cubic, $|E(H_i)|=3|A_i|-l_i=3|B_i|-(3-l_i)$, which implies that $l_i=3$, $E(B_i,S)=\emptyset$ and $|A_i|=|B_i|+1$.
Then $(\cup_{i=1}^tB_i\cup S,\overline{\cup_{i=1}^tB_i\cup S})$ is a bipartition of $G$, contradicting that $G$ is non-bipartite. So (iii) holds.
\end{proof}

\begin{lem}\label{G/X-nice}
Let $G$ be a 2-connected cubic graph. Suppose that $\partial{(X)}$ is a nontrivial tight cut of $G$ such that $G/\overline{X}$ is simple.
If $u\in X$ is nice in $G/\overline{X}$, then $u$ is nice in $G$.
\end{lem}
\begin{proof}
Set $G_1=G/(\overline{X}\rightarrow \overline{x})$ and $G_2=G/X$.
Since $u$ is nice in $G_1$, $G_1-N_{G_1}[u]$ has a perfect matching, say $M_1$.
As $G_1$ is simple and $u\neq \overline{x}$, $(M_1\cup \partial_{G_1}{(u)})\cap \partial_{G_1}{(\overline{x})}$ consists of exactly one edge, say $e_1$, which corresponds to the edge $e_2$ in $G_2$ and the edge $e$ in $G$.
Since $G$ is matching covered by Lemma \ref{matching-covered}, $G_2$ is also matching covered.
So by removing the two ends of $e_2$ from $G_2$, the left graph has a perfect matching, denoted by $M_2$.
If $e_1\in \partial_{G_1}{(u)}$, then $M_1\cup M_2$ is a perfect matching of $G-N_G[u]$; if $e_1\notin \partial_{G_1}{(u)}$, then $(M_1\setminus \{e_1\})\cup M_2\cup \{e\}$ is a perfect matching of $G-N_G[u]$.
Therefore, $u$ is nice in $G$.
\end{proof}

Lemma \ref{G/X-nice} can also be deduced by Lemma 21 (in which $G/\overline{X}$ and $G/X$ are both simple) in \cite{KSS}.

The union of two graphs $G_1$ and $G_2$ is the graph $G_1\cup G_2$ with vertex set $V(G_1)\cup V(G_2)$ and edge set $E(G_1)\cup E(G_2)$.
The following result is a simple observation.
\begin{lem}\label{bi-union}
Let $G_1$ and $G_2$ be two vertex-disjoint bipartite graphs.
If $u_{i1}$ and $u_{i2}$ are two vertices of $G_i$ in different color classes for $i=1,2$, then $(G_1\cup G_2)+\{u_{11}u_{2j},u_{12}u_{2(3-j)}\}$ is a  bipartite graph for $j=1,2$.
\end{lem}

\begin{lem}\label{p-2-cut}
Let $G$ be a 2-connected cubic graph. Assume that $\partial{(X)}=\{ab,cd\}$ is a 2-cut of $G$ such that $a,c\in X$. Then the following statements hold.\\
{\rm (i)} $\partial{(X\setminus\{a\})}$ is a tight cut of $G$,\\
{\rm (ii)} $G[X\setminus\{a,c\}]$ and $G[\overline{X}\setminus \{b,d\}]$ both have perfect matchings,\\
{\rm (iii)} If $ac\notin E(G)$ and $u\in X$ is nice in $G[X]+ac$, then $u$ is nice in $G$,\\
{\rm (iv)} If $G[X]$ is bipartite, then $a$ and $c$ are in distinct color classes of $G[X]$.
\end{lem}
\begin{proof}
(i)\ Since $G$ is cubic, $2|E(G[X])|=\sum_{x\in X}d_{G[X]}(x)=3|X|-2$.
So $|X|$ is even. Lemma \ref{matching-covered} implies that $G$ has a perfect matching, say $M$.
Then $ab\in M$ if and only if $cd\in M$. 
Since $\partial{(X\setminus\{a\})}=(\partial{(a)\setminus \{ab\}})\cup \{cd\}$, $|M\cap\partial{(X\setminus\{a\})}|=1$.
So $\partial{(X\setminus\{a\})}$ is a tight cut of $G$. (Fig. \ref{tight-cut}(b) illustrates an example).

(ii)\ Since $G$ is matching covered (Lemma \ref{matching-covered}), $G$ has a perfect matching containing $ab$, say $M$.
Then $\{ab,cd\}\subseteq M$ by the discussion in (i). So $M\setminus\{ab,cd\}$ is a perfect matching of $G-\{a,b,c,d\}$, and thus $G-\{a,b,c,d\}$ consists of two vertex-disjoint subgraphs $G[X\setminus\{a,c\}]$ and $G[\overline{X}\setminus \{b,d\}]$ which each has a perfect matching.

(iii)\ Set $G'=G[X]+ac$.
If $u\in X\setminus \{a\}$, since $\partial{(X\setminus\{a\})}$ is a tight cut by (i) and $G/(\overline{X}\cup\{a\})$ is simple, $u$ is nice in $G$ by Lemma \ref{G/X-nice}.
If $u=a$, then $N_{G}[a]=(N_{G'}[a]\setminus\{c\})\cup \{b\}$.
Since $u$ is nice in $G'$, $G'-N_{G'}[a]$ has a perfect matching, say $M'$.
On the other hand, (ii) shows that $G[\overline{X}\setminus\{b,d\}]$ has a perfect matching, say $M''$.
Then $M'\cup\{cd\}\cup M''$ is a perfect matching of $G-N_{G}[a]$.
Therefore, $u$ is nice in $G$.

(iv)\ If $ac\in E(G)$, then we are done. So suppose that $ac\notin E(G)$.
Let $(A_X,B_X)$ be a bipartition of $G[X]$ and $l=|E(A_X,\overline{X})|$. Then $l=0,1,2$.
Since $G$ is cubic, $|E(G[X])|=3|A_X|-l=3|B_X|-(2-l)$.
So $l=1$ and $|A_X|=|B_X|$, which implies that $a$ and $c$ are in distinct color classes of $G[X]$.
\end{proof}

\begin{lem}\label{p-2-cut-min}
Let $G$ be a 2-connected non-bipartite cubic graph with a 2-cut.
Then $G$ has a 2-cut $\partial{(X)}=\{ab,cd\}$ such that $a,c\in X$ and $G[X]+ac$ is a 2-connected non-bipartite simple cubic graph with the minimum number of vertices.
\begin{proof}
Assume that $\partial{(Y)}=\{a'b',c'd'\}$ is a 2-cut of $G$ such that $a',c'\in Y$.
We claim that at least one of $G[Y]$ and $G[\overline{Y}]$ is non-bipartite.
Otherwise, by Lemma \ref{p-2-cut}(iv), $a'$ and $c'$ (resp. $b'$ and $d'$) lie in different color classes of $G[Y]$ (resp. $G[\overline{Y}]$).
Then $G$ is bipartite by Lemma \ref{bi-union}, a contradiction.
Choose a 2-cut $\partial{(Y)}$ so that $G[Y]$ is a non-bipartite subgraph with the minimum number of vertices.
Then $a'c'\notin E(G)$. Otherwise, $\partial{(Y\setminus\{a',c'\})}$ is also a 2-cut of $G$ as $G$ is cubic.
Since $G[Y]$ is non-bipartite and $G[\{a',c'\}]$ ($=K_2$) is bipartite, we deduce that $G[Y\setminus\{a',c'\}]$ is non-bipartite by Lemmas \ref{p-2-cut}(iv) and \ref{bi-union}, a contradiction to the minimality of $|Y|$.
So $G[Y]+a'c'$ is a 2-connected non-bipartite simple cubic graph.
Further, $|V(G[Y]+a'c')|$ is minimum.
Otherwise, there exists a 2-cut $\partial{(Z)}=\{a''b'',c''d''\}$ such that $a'',c''\in Z$ and $G[Z]+a''c''$ is a 2-connected non-bipartite simple cubic graph with $|Z|<|Y|$.
Lemma \ref{p-2-cut}(iv) implies that $G[Z]$ is also a non-bipartite subgraph of $G$, contradicting the choice of $\partial{(Y)}$.
Therefore, $\partial{(Y)}$ is a desired 2-cut and the lemma holds.
\end{proof}
\end{lem}

\section{\normalsize Nice vertices in non-bipartite cubic graphs}
Before stating our results, we introduce some necessary notations. Let $G_1$ and $G_2$ be two vertex-disjoint graphs.
Construct two new graphs as follows.
Suppose that $u\in V(G_1)$ and $v\in V(G_2)$ are two vertices with the same degree. Let $\phi$ be a given bijection between $\partial_{G_1}{(u)}$ and $\partial_{G_2}{(v)}$.
The \emph{splicing $G_1$ at $u$ and $G_2$ at $v$}, denoted by $(G_1\odot G_2)_{u,v,\phi}$, is the graph obtained from the union of $G_1-u$ and $G_2-v$ by joining, for each edge $e\in \partial_{G_1}{(u)}$, the end of $e$ in $V(G_1-u)$ to the end of $\phi(e)$ in $V(G_2-v)$.
If $G_1$ is a cubic graph and $G_2\in \{K_4,K_{3,3}\}$, then, up to isomorphism, $(G_1\odot G_2)_{u,v,\phi}$ depends only on the choice of $u$ in $G_1$, and we denote it simply by $(G_1\odot G_2)_u$ (or $G_1\odot G_2$ if the choice of $u$ is irrelevant).
For example, $K_4\odot K_4=\overline{C_6}$, $\overline{C_6}\odot K_4=R_8$ and $K_{3,3}\odot K_4=K_{3,3}^\triangle$ (see Fig. \ref{cubic}(a)-(c)).

\begin{figure}
\centering
\includegraphics{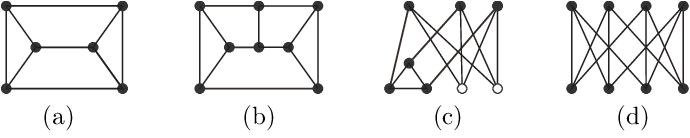}
\caption{\label{cubic} (a)\ $\overline{C_6}$, (b)\ $R_8$, (c)\ $K_{3,3}^\triangle$, (d)\ $H_{4,4}$.}
\end{figure}

Suppose that $e_1=x_1x_2\in E(G_1)$ and $e_2=y_1y_2\in E(G_2)$. The \emph{edge splicing $G_1$ at $e_1$ and $G_2$ at $e_2$}, denoted by $(G_1\oplus G_2)_{e_1,e_2}$ (or simply $(G_1\oplus G_2)_{e_1}$ if no confusion occurs), is the graph obtained from the union of $G_1-e_1$ and $G_2-e_2$ by identifying $x_i$ and $y_i$ into a single vertex for $i=1,2$. 

We can see that $K_{3,3}^\triangle$ has exactly six nice vertices (black) and two non-nice vertices (white); see Fig. \ref{cubic}(c).
Let $\widehat{\mathcal{H}}$ be the set of 3-connected cubic bipartite graphs, $\mathcal{G}_1=\{(K_{3,3}^\triangle\odot H)_{u,v}|u~\text{is not nice in}~K_{3,3}^\triangle$ and $H\in \widehat{\mathcal{H}}\}$ and $\mathcal{G}_2=\{((K_{3,3}^\triangle\odot H_1)_{u_1,v_1}\odot H_2)_{u_2,v_2}|u_i~\text{is not nice in}~K_{3,3}^\triangle,H_i\in \widehat{\mathcal{H}}$ and $i=1,2\}$. For example, see Fig. \ref{splicing}(a) and (b).

\begin{figure}
\centering
\includegraphics{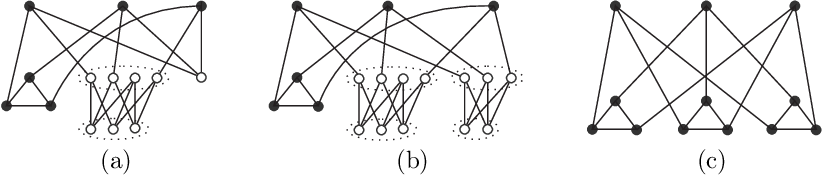}
\caption{\label{splicing} (a)\ A graph in $\mathcal{G}_1$, (b)\ A graph in $\mathcal{G}_2$, (c)\ $((K_{3,3}^\triangle\odot K_4)_u\odot K_4)_v$, where $u$ and $v$ are not nice in $K_{3,3}^\triangle$.}
\end{figure}

Let $L_n$ be a linear quadrangular chain with $n\geq 1$ quadrangles (see Fig. \ref{edge-splicing}(a) for $L_3$) and $\mathcal{H}$ be the set of cubic bipartite graphs.
We refer to an edge with two end vertices of degree 2 as a \emph{22-edge}. 
Set $\mathcal{H^\diamond}=\{(L_n\oplus H)_{e,f}|n\geq 1,e$ is a 22-edge and $H\in \mathcal{H}$\}.
Then every graph in $\mathcal{H^\diamond}$ is bipartite and has precisely one 22-edge.
To replace an edge $e$ of $K_4$ by a graph $H\in \mathcal{H^\diamond}$ is the graph obtained by edge splicing $K_4$ at $e$ and $H$ at the 22-edge.
For $1\leq i\leq 6$, let $\mathcal{F}_i$ be the set of graphs obtained from $K_4$ by replacing $i$ edges of $K_4$ by $i$ graphs in $\mathcal{H^\diamond}$ (not necessarily distinct).
Fig. \ref{edge-splicing}(b) and (c) illustrate an example in $\mathcal{H^\diamond}$ and $\mathcal{F}_2$ respectively.
\begin{figure}
\centering
\includegraphics{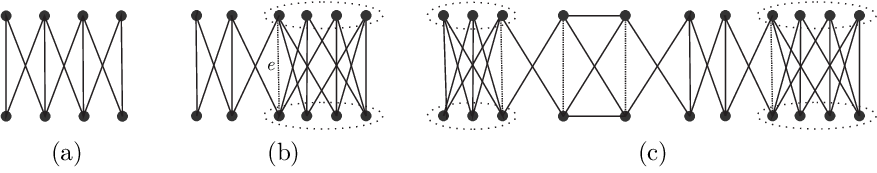}
\caption{\label{edge-splicing} (a)\ $L_3$, (b)\ $(L_2\oplus H_{4,4})_e$, (c)\ A graph in $\mathcal{F}_2$.\ }
\end{figure}

Let $\Upsilon(G)$ denote the number of nice vertices of a cubic graph $G$.
It was shown in \cite{R18} that non-bipartite cubic graphs on at most 6 vertices are $K_4$ and $\overline{C_6}$.
Clearly, $\Upsilon(K_4)=4$, $\Upsilon(\overline{C_6})=6$ and $\Upsilon(K_{3,3}^\triangle)=6$.
Now we can state our main theorem of this section.
\begin{thm}\label{main-1}
Let $G$ be a 2-connected non-bipartite cubic graph. Then\\
{\rm (i)} $\Upsilon(G)\geq 4$. Moreover,  $\Upsilon(G)=4$ if and only if $G\in \{K_4\}\cup (\cup_{i=1}^6\mathcal{F}_i)$.\\
{\rm (ii)} If $G$ is 3-connected and $G\neq K_4$, then $\Upsilon(G)\geq 6$. Moreover, $\Upsilon(G)=6$ if and only if $G\in \{\overline{C_6},K_{3,3}^\triangle\}\cup \mathcal{G}_1\cup\mathcal{G}_2$.
\end{thm}

In order to prove Theorem \ref{main-1}, we first prove several auxiliary results.
A nontrivial barrier $S$ of a graph with a perfect matching is \emph{minimal} if no proper subset of $S$ is also a nontrivial barrier.

\begin{lem}\label{p-3-edge-nice}
Let $G$ be a 3-connected non-bicritical non-bipartite cubic graph.
Then $G$ has a minimal nontrivial barrier $S$ so that each vertex of $S$ is nice.
\end{lem}
\begin{proof}
We proceed by induction on $|V(G)|$.
By Lemma \ref{matching-covered}, $G$ is matching covered.
Since $G$ is non-bicritical, $G$ has a nontrivial barrier by Lemma \ref{Tutte-cor}(i).
Take a minimal nontrivial barrier $S$. Since $G$ is 3-connected, $|S|\geq 3$.
By Lemma \ref{p-3}(iii), $G-S$ has a nontrivial non-bipartite component $K$, which is odd by Lemma \ref{Tutte-cor}(ii).
So $|V(G)|\geq 8$.
If $|V(G)|=8$, then $|S|=3$ and $K$ is a triangle. So $G\cong K_{3,3}^\triangle$ and each vertex of $S$ is nice.
Now assume that $|V(G)|\geq 10$.
Let $G'$ be the graph obtained from $G$ by shrinking each nontrivial component of $G-S$ other than $K$ into a single vertex.

\vspace{8pt}\noindent
{\textbf{Claim 1.} $G_2':=G'/(V(K)\rightarrow k)$ is a brace.}

\begin{proof}
By Lemma \ref{Tutte-cor}(iii), $S$ is independent. So $G_2'$ is bipartite.
By Lemma \ref{p-3}(ii), $G_2'$ is a 3-connected simple cubic graph and so is matching covered.
If $G_2'$ is not a brace, then $G_2'$ has a nontrivial tight cut, say $\partial_{G_2'}{(Y)}$, by Lemma \ref{brick-brace}.
Since $|Y|$ is odd by Lemma \ref{tight-cut-bipartite}, either $Y_+\subseteq S$ or $\overline{Y}_+\subseteq S$, where $\overline{Y}=V(G_2')\setminus Y$. Adjust notation so that $Y_+\subseteq S$.
Then $|Y_+|=|Y\cap S|=|Y_-|+1=|Y\setminus S|+1$ by Lemma \ref{tight-cut-bipartite} again.
Since $\partial_{G_2'}{(Y)}$ is a 3-cut by Lemma \ref{3-cut} and $G_2'$ is also 3-edge-connected, $G_2'[\overline{Y}]$ is connected.
Then $o(G_2'-(S\cap Y))=|Y\setminus S|+1=|S\cap Y|$, which implies $o(G-(S\cap Y))=|S\cap Y|$.
So $S\cap Y$ is a nontrivial barrier of $G$, a contradiction to the minimality of $S$. Hence $G_2'$ is a brace.
\end{proof}

By Lemma \ref{p-3}(ii), $G_1:=G/(\overline{V(K)}\rightarrow \overline{k})$ is a 3-connected non-bipartite simple cubic graph.
Next we consider two cases according to whether $\overline k$ is nice in $G_1$ or not.

{\bf{Case 1.}} $\overline{k}$ is nice in $G_1$.

\begin{figure}
\centering
\includegraphics{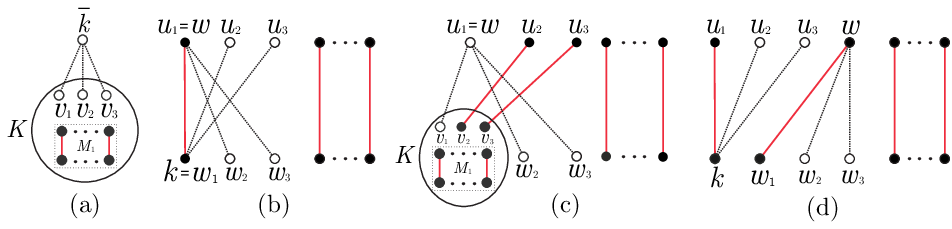}
\caption{\label{pf-3-edge} Illustration for Case 1 in the proof of Lemma \ref{p-3-edge-nice}.}
\end{figure}
We will show that $S$ is a desired barrier. By using Lemmas \ref{p-3}(i), (ii) and \ref{G/X-nice} repeatedly, we only need to show that each vertex of $S$ is nice in $G'$.
By Lemma \ref{p-3}(i), $\partial{(V(K))}$ is a matching of size 3. Set $\partial{(V(K))}=\{u_iv_i|u_i\in S,v_i\in V(K)~\text{and}~i=1,2,3\}$.
Let $w$ be a vertex in $S$ and $N_{G'}(w)=\{w_1,w_2,w_3\}$.
By Claim 1 and Lemma \ref{brace}, $G_2'-\{u_2,u_3,w_2,w_3\}$ has a perfect matching $M_2$.
Since $\overline{k}$ is nice in $G_1$, $G_1-N_{G_1}[\overline{k}]$ has a perfect matching $M_1$ (see Fig. \ref{pf-3-edge}(a)).
If $w\in \{u_1,u_2,u_3\}$, assume that $w=u_1$ and $w_1=v_1$, then $wk\in M_2$ (see Fig. \ref{pf-3-edge}(b)) and $G'-N_{G'}[w]=(G_1-\{\overline{k},v_1\})\cup(G_2'-\{w,k,w_2,w_3\})\cup \{u_2v_2,u_3v_3\}$.
So $M_1\cup(M_2\setminus \{wk\})\cup \{u_2v_2,u_3v_3\}$ is a perfect matching of $G'-N_{G'}[w]$ (see Fig. \ref{pf-3-edge}(c)).
If $w\notin \{u_1,u_2,u_3\}$, then $\{u_1k,ww_1\}\subseteq M_2$ (see Fig. \ref{pf-3-edge}(d)) and $G'-N_{G'}[w]=(G_1-\overline{k})\cup(G_2'-(\{k\}\cup N_{G'}[w]))\cup \partial{(V(K))}$.
So $M_1\cup(M_2\setminus \{u_1k,ww_1\})\cup \partial{(V(K))}$ is a perfect matching of $G'-N_{G'}[w]$.
Hence $w$ is nice in $G'$. By the arbitrariness of $w$, each vertex of $S$ is nice in $G'$.

{\bf{Case 2.}} $\overline{k}$ is not nice in $G_1$.

Lemma \ref{bicritical} implies that $G_1$ is not bicritical.
So $G_1$ has a minimal nontrivial barrier $S_1$ so that each vertex of $S_1$ is nice in $G_1$ by the induction hypothesis.
Then $\overline{k}\notin S_1$ and $S_1\subseteq V(K)$.
Since $\partial{(V(K))}$ is a tight cut of $G$ by Lemma \ref{p-3}(i), each vertex of $S_1$ is also nice in $G$ by Lemma \ref{G/X-nice}.
Since $G_2:=G/((V(K)\rightarrow k)$ is 3-connected (Lemma \ref{p-3}(ii)), $G_2-k$ is connected.
Then $G_2-k$ is contained in some component of $G-S_1$, which implies $o(G-S_1)=o(G_1-S_1)$. So $S_1$ is also a nontrivial barrier of $G$.
If $S_1$ properly contains a nontrivial barrier $T$ of $G$, then $o(G-T)=o(G_1-T)$ by the same reason as above.
Hence $T$ is also a barrier of $G_1$, contradicting the minimality of $S_1$.
So $S_1$ is a minimal nontrivial barrier of $G$, which is a desired barrier.
\end{proof}

\begin{lem}\label{nice-number}
Let $G$ be a 2-connected non-bipartite cubic graph. Then $\Upsilon(G)\geq 4$. Moreover, $\Upsilon(G)\geq 6$ if $G$ is 3-connected and $G\neq K_4$.
\end{lem}
\begin{proof}
We use induction on $|V(G)|$. If $|V(G)|=4$, then $G=K_4$ and $\Upsilon(G)=4$.
Suppose that $|V(G)|\geq 6$ and the result holds for graphs with fewer vertices than $|V(G)|$.
If $G$ is bicritical, then $\Upsilon(G)=|V(G)|\geq6$ by Lemma \ref{bicritical} and the result holds.
So suppose that $G$ is not bicritical. First assume that $G$ is 3-connected.
By Lemma \ref{p-3-edge-nice}, $G$ has a minimal nontrivial barrier $S$ so that each vertex of $S$ is nice.
Since $G$ is 3-connected, $|S|\geq 3$.
Further, since $G$ is non-bipartite, $G-S$ contains a nontrivial non-bipartite component $K$ by Lemma \ref{p-3}(iii).
Since $\partial{(V(K))}$ is a tight cut by Lemma \ref{p-3}(i), $\Upsilon(G)\geq \Upsilon(G/\overline{V(K)})-1+|S|\geq \Upsilon(G/\overline{V(K)})+2$ by using Lemma \ref{G/X-nice}.
By Lemma \ref{p-3}(ii), $G/\overline{V(K)}$ is a 3-connected non-bipartite simple cubic graph.
So $\Upsilon(G/\overline{V(K)})\geq 4$ by the induction hypothesis and thus $\Upsilon(G)\geq 6$.

Now suppose that $G$ is not 3-connected. Since a cubic graph is 3-connected if and only if it is 3-edge-connected, $G$ has a 2-cut.
By Lemma \ref{p-2-cut-min}, we can take a 2-cut $\partial{(X)}=\{ab,cd\}$ such that $a,c\in X$ and $G[X]+ac$ is a 2-connected non-bipartite simple cubic graph. 
Then $\Upsilon(G[X]+ac)\geq 4$ by the induction hypothesis.
Further, as $ac\notin E(G)$, $\Upsilon(G)\geq \Upsilon(G[X]+ac)\geq 4$ by using Lemma \ref{p-2-cut}(iii), and so the proof is completed.
\end{proof}

In the following, we show that the lower bounds in Lemma \ref{nice-number} are sharp.

\begin{lem}\label{6-nice}
Let $G$ be a 3-connected non-bipartite cubic graph.
Then $\Upsilon(G)=6$ if and only if $G\in \{\overline{C_6},K_{3,3}^\triangle\}\cup \mathcal{G}_1\cup\mathcal{G}_2$.
\end{lem}
\begin{proof}
Sufficiency. If $G\in\{\overline{C_6},K_{3,3}^\triangle\}$, then $\Upsilon(G)=6$ clearly.
Assume that $G=(K_{3,3}^\triangle\odot H)_{u,v}\in \mathcal{G}_1$, where $H\in \mathcal{\widehat{H}}$.
Let $H=H(A_H,B_H)$ and $v\in B_H$.
Then $A_H$ (resp. $B_H\cup N_{K_{3,3}^\triangle}(u)$) is a barrier of $G$ such that each vertex of $G$ in $B_H\setminus\{v\}$ (resp. $A_H\cup \{u'\}$) is an isolated vertex, where $u'$ ($\neq u$) is the other non-nice vertex of $K_{3,3}^\triangle$.
By Theorem \ref{criteria}, each vertex of $G$ in $(V(H)\setminus\{v\})\cup\{u'\}$ is not nice.
Since $\partial{(\overline{V(H-v)})}$ is a tight cut of $G$ and $G/V(H-v)\cong K_{3,3}^\triangle$, the 6 nice vertices of $K_{3,3}^\triangle$ are also nice in $G$ by Lemma \ref{G/X-nice}.
So $\Upsilon(G)=6$. Similarly, we can prove that $\Upsilon(G)=6$ if $G\in \mathcal{G}_2$.
Hence the sufficiency holds.

Necessity. If $G$ is bicritical, then $|V(G)|=6$ by Lemma \ref{bicritical} and so $G=\overline{C_6}$.
Next suppose that $G$ is not bicritical.
By Lemma \ref{p-3-edge-nice}, $G$ has a minimal nontrivial barrier $S$ so that each its vertex is nice.
Lemma \ref{p-3}(iii) implies that $G-S$ contains a nontrivial non-bipartite component $K$.
Then $G/\overline{V(K)}$ is a 3-connected non-bipartite simple cubic graph by Lemma \ref{p-3}(ii).
Combining Lemmas \ref{p-3}(i) and \ref{G/X-nice}, $\Upsilon(G)\geq\Upsilon(G/\overline{V(K)})-1+|S|\geq\Upsilon(G/\overline{V(K)})+2$, where the second inequality is because $G$ is 3-connected.
As $\Upsilon(G)=6$, $\Upsilon(G/\overline{V(K)})\leq4$.
By Lemma \ref{nice-number}, $G/\overline{V(K)}=K_4$, which implies that $K$ is a triangle.
Since all vertices of $K$ are nice in $G/\overline{V(K)}$, $G$ has exactly 3 nice vertices in $K$ by Lemma \ref{G/X-nice}.
So $|S|=3$ and $G$ has no nice vertices in $G-S-V(K)$.
By Lemmas \ref{matching-covered} and \ref{Tutte-cor}(ii), $G-S$ has no even components.
As $o(G-S)=|S|=3$, $G-S$ has precisely 3 components.
We claim that the other nontrivial components of $G-S$ is bipartite.
Otherwise, say $K'$ is another nontrivial non-bipartite component of $G-S$.
Then we have that $K'$ contains at least 3 nice vertices of $G$ by Lemmas \ref{p-3}(i), (ii), \ref{nice-number} and \ref{G/X-nice}, a contradiction.
If $K$ is the unique nontrivial component of $G-S$, then $G\cong K_{3,3}^\triangle$.
Assume that $G-S$ has exactly two nontrivial components $K$ and $K^*$.
Set $G_1=G/(V(K^*)\rightarrow k^*)$ and $G_2=G/(\overline{V(K^*)}\rightarrow \overline{k^*})$.
Then we have that $G_1\cong K_{3,3}^\triangle$, $G_2\in \mathcal{\widehat{H}}$ by Lemma \ref{p-3}(ii) and $k^*$ is not nice in $G_1$.
So $G=(G_1\odot G_2)_{k^*,\overline{k^*}}\in \mathcal{G}_1$.
Similarly, $G\in \mathcal{G}_2$ if $G-S$ consists of three nontrivial components. Hence $G\in \{\overline{C_6},K_{3,3}^\triangle\}\cup \mathcal{G}_1\cup\mathcal{G}_2$.
\end{proof}

\begin{lem}\label{4-nice}
Let $G$ be a 2-connected non-bipartite cubic graph. Then $\Upsilon(G)=4$ if and only if $G\in \{K_4\}\cup (\cup_{i=1}^6\mathcal{F}_i)$.
\end{lem}
\begin{proof}
If $|V(G)|=4$, then $G=K_4$ and the result holds. So assume that $|V(G)|\geq 6$.

Sufficiency. Assume that $G$ is a graph in $\mathcal{F}_i$ for some $1\leq i\leq 6$, which is obtained by replacing $i$ edges $e_1,\ldots,e_i$ of $K_4$ by $i$ graphs $H_1^\diamond,\ldots,H_i^\diamond$ in $\mathcal{H^\diamond}$ respectively.
For $1\leq j\leq i$, let $(X_{j1},X_{j2})$ be the bipartition of $H_j^\diamond$ and $x_{j1}x_{j2}$ be its 22-edge with $x_{jl}\in X_{jl}$ for $l=1,2$, and $e_j=y_{j1}y_{j2}$.
Denote by $z_{jl}$ the vertex obtained by identifying $x_{jl}$ and $y_{jl}$ for $l=1,2$.
Then $(X_{j1}\setminus\{x_{j1}\})\cup\{z_{j1}\}$ is a barrier of $G$ such that vertices of $G$ in $X_{j2}\setminus\{x_{j2}\}$ are isolated vertices, which  are not nice in $G$ by Theorem \ref{criteria}. By symmetry, vertices of $G$ in $X_{j1}\setminus\{x_{j1}\}$ are not nice.
Moreover, by using Lemma \ref{p-2-cut}(iii) $i$ times, we obtain that the 4 vertices of $G$ in $\overline{\cup_{j=1}^iV(H_j^\diamond)\setminus\{x_{j1},x_{j2}\}}$ are nice, and so $\Upsilon(G)=4$.

Necessity. We first prove the following claim.

\vspace{8pt}\noindent
{\textbf{Claim 2.}} $G$ has a 2-cut $\partial{(X)}=\{ab,cd\}$ such that $a,c\in X$ and\\
(i) $G_1:=G[X]+ac$ is a 2-connected non-bipartite simple cubic graph with the minimum number of vertices,\\
(ii) $\Upsilon(G_1)=4$, and either $G_1\cong K_4$ or $G_1$ has a 2-cut,\\
(iii) $G_2:=G[\overline{X}\cup\{a,c\}]+ac\in \mathcal{H^\diamond}$.
\begin{proof}
From Lemma \ref{nice-number}, $G$ has a 2-cut.
By Lemma \ref{p-2-cut-min}, we can take a 2-cut $\partial{(X)}=\{ab,cd\}$ of $G$ such that $a,c\in X$ and Claim 2(i) holds.
Then $ac\notin E(G)$. Using Lemma \ref{p-2-cut}(iii), $\Upsilon(G)\geq \Upsilon (G_1)$.
As $\Upsilon(G)=4$, Lemma \ref{nice-number} implies that $\Upsilon(G_1)=4$, and either $G_1\cong K_4$ or $G_1$ has a 2-cut.
So Claim 2(ii) holds.
It remains to prove Claim 2(iii).
Since $G$ is cubic, $G_2=(L_n\oplus H)_{e,f}$ for some $n\geq1$ and a 2-connected cubic graph $H$, where $e$ is a 22-edge of $L_n$.
If $H$ is non-bipartite, then $\Upsilon(H)\geq 4$ by Lemma \ref{nice-number}, and so $\Upsilon(G)\geq\Upsilon(G_1)+\Upsilon(G_2)\geq 8$ by Lemma \ref{p-2-cut}(iii), a contradiction.
Thus $H$ is bipartite and $G_2\in \mathcal{H}^\diamond$.
\end{proof}

By Claim 2, take a 2-cut $\partial{(X)}=\{ab,cd\}$ of $G$ such that $a,c\in X$ and Claim 2(i)-(iii) holds.
We proceed by induction on the number of 2-cuts of $G$.
If $\partial{(X)}$ is the unique 2-cut of $G$, then either $G_1\cong K_4$ or $G_1$ has a 2-cut by Claim 2(ii).
If $G_1$ has a 2-cut, say $\partial_{G_1}{(Y)}$, then $G$ has another 2-cut $C$, where $C=\partial{(Y)}$ if $ac\notin E(G_1[Y])$, and $C=\partial{(\overline{Y})}$ otherwise, a contradiction.
So $G_1\cong K_4$ and $G$ is obtained by replacing exactly one edge in $G_1$ by the graph $G_2$ in $\mathcal{H}^\diamond$ by Claim 2(iii).
Thus $G\in \mathcal{F}_1$ and the result holds.
Assume that the result holds for graphs with less 2-cuts than $G$.
By Claim 2(ii) again, if $G_1\cong K_4$, then $G\in \mathcal{F}_1$ due to Claim 2(iii), and otherwise $G_1\in \cup_{i=1}^6\mathcal{F}_i$ by induction.
Assume that $G_1\in \mathcal{F}_i$ for some $1\leq i\leq 6$, which is obtained by replacing $i$ edges $f_1,f_2,\ldots,f_i$ of $K_4$ by $i$ graphs $H_1,H_2,\ldots,H_i$ in $\mathcal{H^\diamond}$, respectively.
Let $U_j$ be the set of vertices in $H_j$ with degree 3 and $V_j=V(G_1)\setminus U_j$ for $1\leq j\leq i$.
We claim that $ac\in E(K_4)$. Otherwise, $ac\in \partial_{G_1}{(V_j)}\cup E(G_1[U_j])$ for some $1\leq j\leq i$.
Since $G_1[U_j]$ is bipartite and $G_1$ is non-bipartite, $G_1[V_j]$ ($=G[V_j]$) is non-bipartite by Lemma \ref{bi-union}.
Then $\partial{(V_j)}$ is a 2-cut of $G$ and $G[V_j]+f_j'$ is a 2-connected non-bipartite simple cubic graph with less vertices than $G_1$, a contradiction to Claim 2(i), where $f_j'$ is the edge corresponding to the edge $f_j$ in $K_4$.
So $G_1\in \mathcal{F}_{i}$ for some $1\leq i\leq 5$ and $G\in \mathcal{F}_{i+1}\subseteq \cup_{i=2}^6\mathcal{F}_i$, which completes the proof.
\end{proof}

Consequently, Theorem \ref{main-1} is obtained by Lemmas \ref{nice-number}-\ref{4-nice}.

\section{\normalsize Nice pairs in bipartite cubic graphs}
In this section, we discuss nice pairs of vertices in cubic bipartite  graphs.
Let $\mathcal{T}$ be the set of cubic bipartite graphs defined as follows: (i) $K_{3,3}\in \mathcal{T}$, (ii) if $T\in \mathcal{T}$,
then $(T\oplus (L_n\oplus K_{3,3})_{e_1})_{f,e_2}\in\mathcal{T}$, where $e_1$ and $e_2$ are two 22-edges of $L_n$ and $n\geq1$.
For a cubic bipartite graph $G(A,B)$, denote by $(a,b)$ (resp. $(A',B'))$ a pair of vertices $a\in A$ and $b\in B$ (resp. vertex subsets $A'\subseteq A$ and $B'\subseteq B$). Then our main results are as follows.

\begin{thm}\label{main-2}
Let $G(A,B)$ be a cubic bipartite graph.
Then $G$ has a nice pair set $(A',B')$ with $|A'|\geq 3$ and $|B'|\geq 3$.
Moreover, for each nice pair set $(A',B')$, $|A'|\leq3$ and $|B'|\leq3$ if and only if $G\in \mathcal{T}$.
\end{thm}

\begin{cor}\label{cor}
Let $G$ be a cubic bipartite graph.
Then $G$ has at least 9 nice pairs of vertices, with equality if and only if $G=K_{3,3}$.
\end{cor}}

Before proving Theorem \ref{main-2} and Corollary \ref{cor}, we show some useful results.
Since the edge chromatic number of a bipartite graph is equal to its maximum degree (see \cite{BM}, Theorem 17.2), every cubic bipartite graph has 3 edge-disjoint perfect matchings. This implies the following result.

\begin{lem}\label{bi-matching-covered}
Every cubic bipartite graph is matching covered.
\end{lem}

\begin{prop}\label{brace-nice}
Let $G(A,B)$ be a cubic bipartite graph.
Then $G$ is a brace if and only if each pair $(a,b)$ of $G$ is nice (i.e. $(A,B)$ is a nice pair set).
\end{prop}
\begin{proof}
Necessity. Set $N_G(a)=\{a_1,a_2,a_3\}$ and $N_G(b)=\{b_1,b_2,b_3\}$.
By Lemma \ref{brace}, $G-\{a_2,a_3,b_2,b_3\}$ has a perfect matching, denoted by $M$.
If $a\notin N_G(b)$, then $\{aa_1,bb_1\}\subseteq M$ since $G$ is cubic.
It follows that $M\setminus \{aa_1,bb_1\}$ is a perfect matching of $G-N_G[a]-N_G[b]$.
If $a\in N_G(b)$, assume that $a=b_1$ and $b=a_1$, then $ab\in M$ and so $M\setminus \{ab\}$ is a perfect matching of $G-N_G[a]-N_G[b]$.
Thus $(a,b)$ is a nice pair and the necessity holds.

Sufficiency. Suppose to the contrary that $G$ is not a brace.
Since $G$ is matching covered (Lemma \ref{bi-matching-covered}) and bipartite, $G$ has a nontrivial tight cut by Lemma \ref{brick-brace}. Take a nontrivial tight cut  $\partial{(X)}$ so that $X_+\subseteq A$.
By Lemma \ref{tight-cut-bipartite}, $|A\cap X|=|B\cap X|+1$ and $E(B\cap X,A\cap \overline{X})=\emptyset$.
For $a\in A\cap \overline{X}$ and $b\in B\cap X$, since $(a,b)$ is nice in $G$, $G-N_G[a]-N_G[b]$ has a perfect matching, say $F$.
As $|(A\cap X)\setminus N_G(b)|=|A\cap X|-3=|B\cap X|-2<|(B\cap X)\setminus \{b\}|$, $F\cap E(B\cap X,A\cap \overline{X})\neq \emptyset$, a contradiction. So $G$ is a brace.
\end{proof}

\begin{lem}\label{H/X-nice-pair}
Let $G(A,B)$ be a cubic bipartite graph.
Assume that $\partial{(X)}$ is a nontrivial tight cut of $G$ such that $|A\cap X|=|B\cap X|+1$ and
$G/\overline{X}$ is simple. For $a\in A$ and $b\in B$,\\
{\rm (i)} if $a,b\in X$, then $(a,b)$ is a nice pair of $G/\overline{X}$ if and only if it is a nice pair of $G$,\\
{\rm (ii)} Assume that $G$ is 3-connected, $a\in X$ and $b\in \overline{X}$.
If $(a,\overline{x})$ and $(x,b)$ are nice pairs of $G/(\overline{X}\rightarrow \overline{x})$ and $G/(X\rightarrow x)$ respectively, then $(a,b)$ is a nice pair of $G$.
\end{lem}
\begin{proof}
Set $G_1=G/(\overline{X}\rightarrow \overline{x})$ and $G_2=G/(X\rightarrow x)$.
For an edge $e\in \partial{(X)}$, denote by $e_i$ the edge in $G_i$ corresponding to $e$ for $i=1,2$.

(i)\ Necessity. Since $(a,b)$ is a nice pair of $G_1$, $G_1-N_{G_1}[a]-N_{G_1}[b]$ has a perfect matching, say $M_1$.
As $G_1$ is simple and $b\neq \overline{x}$, $(M_1\cup\partial_{G_1}{(a)})\cap \partial_{G_1}{(\overline{x})}$ consists of exactly one edge, say $e_1$.
Since $G$ is matching covered by Lemma \ref{bi-matching-covered}, $G_2$ is also matching covered.
So the graph obtained from $G_2$ by removing the two end-vertices of $e_2$ has a perfect matching, denoted by $M_2$.
If $e_1\in M_1$, then $G-N_G[a]-N_G[b]=(G_1-N_{G_1}[a]-N_{G_1}[b]-\{\overline{x}\})\cup (G_2-x)$, and so $(M_1\setminus \{e_1\})\cup M_2\cup \{e\}$ is a perfect matching of $G-N_G[a]-N_G[b]$. 
If $e_1\notin M_1$, then $M_1\cup M_2$ is a perfect matching of $G-N_G[a]-N_G[b]$.
Thus $(a,b)$ is a nice pair of $G$.

Sufficiency. Let $M$ be a perfect matching of $G-N_G[a]-N_G[b]$. Then $(M\cup \partial{(a)})\cap \partial{(X)}$ consists of precisely one edge, say $e'$.
If $e'\in M$, then $(M\cap E(G_1))\cup \{e_1'\}$ is a perfect matching of $G_1-N_{G_1}[a]-N_{G_1}[b]$.
If $e'\notin M$, then $M\cap E(G_1)$ is a perfect matching of $G_1-N_{G_1}[a]-N_{G_1}[b]$.
So $(a,b)$ is a nice pair of $G_1$.
\begin{figure}
\centering
\includegraphics{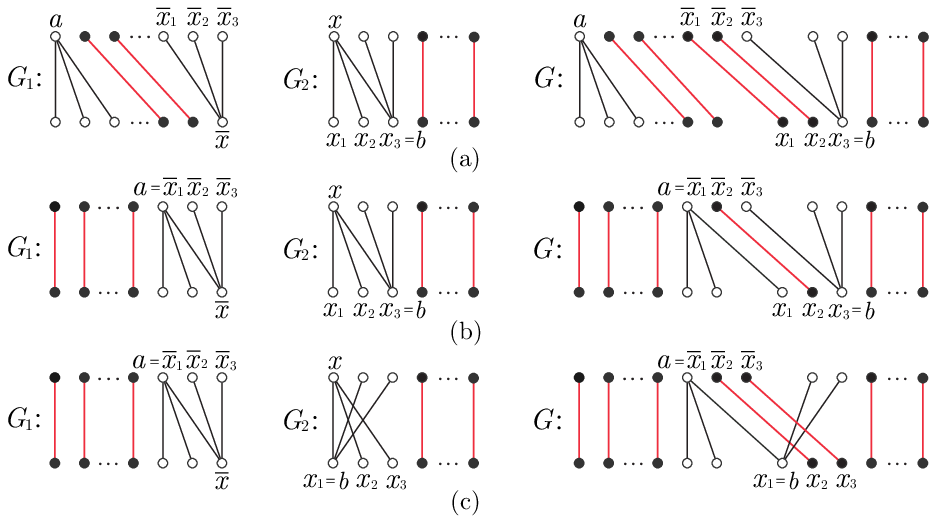}
\caption{\label{bi-nice} Illustration for the proof of Lemma \ref{H/X-nice-pair}(ii).}
\end{figure}

(ii)\ By Lemmas \ref{cut-matching} and \ref{3-cut}, $\partial{(X)}$ is a matching of size 3.
Assume that $\partial{(X)}=\{\overline{x}_ix_i|\overline{x}_i\in A,x_i\in B~\text{and}~i=1,2,3\}$.
Then $\overline{x}_i\overline{x}\in E(G_1)$ and $xx_i\in E(G_2)$ for $i=1,2,3$.
Let $F_1$ and $F_2$ be perfect matchings of $G_1-N_{G_1}[a]-N_{G_1}[\overline{x}]$ and $G_2-N_{G_2}[x]-N_{G_2}[b]$, respectively.
If $a\notin N_{G_1}(\overline{x})$ and $b\notin N_{G_2}(x)$, then obviously $F_1\cup F_2\cup \partial{(X)}$ is a perfect matching of $G-N_{G}[a]-N_{G}[b]$.
If $a\notin N_{G_1}(\overline{x})$ and $b\in N_{G_2}(x)$, assume that $b=x_3$, then $G-N_{G}[a]-N_{G}[b]=(G_1-N_{G_1}[a]-\{\overline{x}_3,\overline{x}\})\cup(G_2-N_{G_2}[b])$, and so $F_1\cup F_2\cup \{\overline{x}_1x_1,\overline{x}_2x_2\}$ is a perfect matching of $G-N_{G}[a]-N_{G}[b]$ (see Fig. \ref{bi-nice}(a)).
If $a\in N_{G_1}(\overline{x})$ and $b\notin N_{G_2}(x)$, then $G-N_{G}[a]-N_{G}[b]$ also has a perfect matching by symmetry.
Now assume that $a\in N_{G_1}(\overline{x})$ and $b\in N_{G_2}(x)$.
If $ab\notin E(G)$, say $a=\overline{x}_1$ and $b=x_3$, then $G-N_{G}[a]-N_{G}[b]=(G_1-N_{G_1}[a]-\overline{x}_3)\cup (G_2-N_{G_2}[b]-x_1)$, and so $F_1\cup F_2\cup \{\overline{x}_2x_2\}$ is a perfect matching of $G-N_{G}[a]-N_{G}[b]$ (see Fig. \ref{bi-nice}(b)).
Similarly, if $ab\in E(G)$, say $a=\overline{x}_1$ and $b=x_1$, then $F_1\cup F_2\cup \{\overline{x}_2x_2,\overline{x}_3x_3\}$ is a perfect matching of $G-N_{G}[a]-N_{G}[b]$ (see Fig. \ref{bi-nice}(c)). Therefore, $(a,b)$ is a nice pair of $G$.
\end{proof}

Note that $K_{3,3}$ and $H_{4,4}$ (see Fig. \ref{cubic}(d)) are all 3-connected cubic bipartite graphs on at most 8 vertices (also see \cite{R18}).
\begin{lem}\label{bi-3-edge}
Let $G(A,B)$ be a 3-connected cubic bipartite graph. 
If $|V(G)|\geq 10$, then $G$ has a nice pair set $(A',B')$ with $|A'|\geq 3$ and $|B'|\geq 5$, and otherwise $(A,B)$ is a nice pair set of $G$.
\end{lem}
\begin{proof}
If $|V(G)|<10$, then $G\in \{K_{3,3},H_{4,4}\}$ and each pair of vertices is nice.
So $(A,B)$ is a nice pair set and the latter statement holds.
Next assume that $|V(G)|\geq10$.
We proceed by induction on the number of nontrivial tight cuts of $G$. 
If $G$ has no nontrivial tight cuts, since $G$ is matching covered (Lemma \ref{bi-matching-covered}) and bipartite, $G$ is a brace by Lemma \ref{brick-brace}.
Since $|A|=|B|\geq 5$, the result follows from Proposition \ref{brace-nice}.
Assume that $G$ has a nontrivial tight cut and the result holds for graphs with fewer nontrivial tight cuts than $G$.
Lemma \ref{brick-brace} implies that $G$ is not a brace.
By Lemma \ref{non-brace}, $G$ has a nontrivial tight cut $\partial{(X)}$ so that one of the two $\partial{(X)}$-contractions is a brace.
Adjust notation so that $G_1:=G/(\overline{X}\rightarrow \overline{x})$ is a brace.
By Lemmas \ref{G/X}-\ref{3-cut}, $G_2:=G/(X\rightarrow x)$ is a 3-connected simple cubic bipartite graph.
Set $G_i=G_i(A_i,B_i)$, $i=1,2$.
By Proposition \ref{brace-nice}, $(A_1,B_1)$ is a nice pair set of $G_1$ with $|A_1|=|B_1|\geq 3$.
If $G_2$ is a brace, then $(A_2,B_2)$ is also a nice pair set of $G_2$ with $|A_2|=|B_2|\geq 3$.
Using Lemma \ref{H/X-nice-pair}, $(A_i,B)$ is a nice pair set of $G$ with $|A_i|\geq 3$ and $|B|\geq 5$ for some $i=1,2$.
If $G_2$ is not a brace, then $|V(G_2)|\geq 10$. 
By the induction hypothesis, $G_2$ has a nice pair set $(A_2',B_2')$ with $|A_2'|\geq 3$ and $|B_2'|\geq 5$.
If $x\in A_2'$, then $\overline{x}\in B_1$ and $(A_1,B_2')$ is a nice pair set of $G$ by Lemma \ref{H/X-nice-pair}(ii).
If $x\in B_2'$, then $\overline{x}\in A_1$ and $(A_2',B_1\cup (B_2'\setminus \{x\}))$ is a nice pair set of $G$ with $|A_2'|\geq 3$ and $|B_1\cup (B_2'\setminus \{x\})|\geq 3+4=7$ by Lemma \ref{H/X-nice-pair}.
If $x\notin A_2'$ and $x\notin B_2'$, then $(A_2',B_2')$ is a nice pair set of $G$ by Lemma \ref{H/X-nice-pair}(i), and so the lemma holds.
\end{proof}

\begin{lem}\label{bi-2-cut}
Assume that $\partial{(X)}=\{uv,wz\}$ is a 2-cut of a cubic bipartite graph $G(A,B)$ such that $u,w\in X$ and $uw\notin E(G)$.
Then the following two statements hold.\\
{\rm (i)} If $(a,b)$ is a nice pair of $G$, then either $\{a,b\}\subseteq X$ or $\{a,b\}\subseteq \overline{X}$,\\
{\rm (ii)} If $\{a,b\}\subseteq X$, then $(a,b)$ is a nice pair of $G$ if and only if it is nice in $G[X]+uw$.
\end{lem}
\begin{proof}
Since $G[X]$ ($\subseteq G$) is bipartite, $u$ and $w$ are in different color classes of $G$ by Lemma \ref{p-2-cut}(iv).
Without loss of generality, assume that $u\in A$. Then $z\in A$ and $v,w\in B$.

(i)\ Otherwise, suppose that $a\in X$ and $b\in \overline{X}$.
Since $(a,b)$ is nice in $G$, $G-N_G[a]-N_G[b]$ has a perfect matching, denoted by $M$.
Let $A_1=A\cap X$ and $B_1=B\cap X$. Since $G$ is cubic, $|E(G[X])|=3|A_1|-1=3|B_1|-1$, and so $|A_1|=|B_1|$.
If $a\neq u$, then $|M\cap E(A_1,\overline{X})|=2$ since $|B_1\setminus N_G(a)|=|B_1|-3=|A_1\setminus \{a\}|-2$, which is impossible as $E(A_1,\overline{X})=\{uv\}$.
If $a=u$, then $uv\notin M$ and $|M\cap E(A_1,\overline{X})|=1$, which is also impossible.
So (i) holds.

(ii)\ Set $G_1=G[X]+uw$. Then $G_1$ is isomorphic to $G/(\overline{X}\cup \{u\})$ and $G/(\overline{X}\cup \{w\})$.
By Lemma \ref{p-2-cut}(i), $\partial{(X\setminus \{u\})}$ and $\partial{(X\setminus \{w\})}$ are tight cuts of $G$.
If $a\neq u$ or $b\neq w$, then $(a,b)$ is a nice pair of $G$ if and only if it is nice in $G_1$ by Lemma \ref{H/X-nice-pair}(i).
So we only need to show that the result holds for $a=u$ and $b=w$.
If $(u,w)$ is a nice pair of $G$, then $G-N_G[u]-N_G[w]$ has a perfect matching, say $F$.
It follows that $F\cap E(G_1)$ is a perfect matching of $G_1-N_{G_1}[u]-N_{G_1}[w]$. So $(u,w)$ is nice in $G_1$.
Conversely, if $(u,w)$ is nice in $G_1$, then $G_1-N_{G_1}[u]-N_{G_1}[w]$ has a perfect matching, say $F_1$.
Note that $G-N_G[u]-N_G[w]=(G_1-N_{G_1}[u]-N_{G_1}[w])\cup (G[\overline{X}\setminus\{v,z\}])$.
By Lemma \ref{p-2-cut}(ii), $G[\overline{X}\setminus\{v,z\}]$ has a perfect matching, denoted by $F_2$.
Then $F_1\cup F_2$ is a perfect matching of $G-N_G[u]-N_G[w]$. So $(u,w)$ is nice in $G$, and the proof is completed.
\end{proof}

\begin{lem}\label{2-cut-min}
Let $G(A,B)$ be a cubic bipartite graph with a 2-cut. Then\\
{\rm (i)} $G$ has a 2-cut $\partial{(X)}=\{uv,wz\}$ such that $u,w\in X$, $G[X]+uw\in \mathcal{\widehat{H}}$ and $G[\overline{X}\cup\{u,w\}]+uw\in \mathcal{H}^\diamond$,\\
{\rm (ii)} $G$ has a nice pair set $(A',B')$ with $|A'|\geq 3$ and $|B'|\geq 3$.
\end{lem}
\begin{proof}
Take a 2-cut $\partial{(X)}=\{uv,wz\}$ of $G$ such that $|X|$ is minimum.
Then $uw\notin E(G)$. Otherwise, $\partial{(X\setminus\{u,w\})}$ is also a 2-cut of $G$ as $G$ is cubic with $|X\setminus\{u,w\}|<|X|$, a contradiction.
Set $G_1=G[X]+uw$ and $G_2=G[\overline{X}\cup\{u,w\}]+uw$.
Then $G_1$ is a simple cubic graph and $G_2=(L_n\oplus H)_{e,f}$ for some $n\geq 1$, where $e$ is a 22-edge and $H$ is a cubic graph.
Lemma \ref{p-2-cut}(iv) implies that $G_1$ and $H$ are bipartite.
So $G_2\in \mathcal{H}^\diamond$.
If $G_1$ is not 3-connected, then $G_1$ has a 2-cut $\partial_{G_1}{(Y)}=\{u_1v_1,w_1z_1\}$ such that $u_1,w_1\in Y$ and $u_1w_1\notin E(G_1)$.
But thus $\partial{(Y)}$ is a 2-cut of $G$ if $ac\notin E(G_1[Y])$; $\partial{(\overline{Y})}$ is a 2-cut of $G$ otherwise, a contradiction to the minimality of $|X|$.
So $G_1$ is 3-connected.
 Hence $G_1\in \mathcal{\widehat{H}}$, and (i) holds.
By Lemma \ref{bi-3-edge}, we obtain that $G_1$ has a nice pair set $(A',B')$ with $|A'|\geq 3$ and $|B'|\geq 3$, which is also nice in $G$ by Lemma \ref{bi-2-cut}(ii). So (ii) holds.
\end{proof}

Next we show that the lower bounds on $|A'|$ and $|B'|$ of Theorem \ref{main-2} are best possible.
\begin{lem}\label{best}
Let $G(A,B)$ be a cubic bipartite graph and $(A',B')$ be a nice pair set of $G$. Then $|A'|\leq3$ and $|B'|\leq3$ if and only if $G\in \mathcal{T}$.
\end{lem}
\begin{proof}
Necessity. We use induction on $|V(G)|$.
If $|V(G)|=6$, then $G=K_{3,3}\in \mathcal{T}$, and so the result holds.
Assume that $|V(G)|\geq 8$ and the result holds for graphs with less vertices than $G$.
If $G$ is 3-connected, then $|V(G)|<10$ by Lemma \ref{bi-3-edge}, and so $G=K_{3,3}\in \mathcal{T}$.
Suppose that $G$ is not 3-connected. Then $G$ has a 2-cut. 
By Lemma \ref{2-cut-min}(i), we can take a 2-cut $\partial{(X)}=\{uv,wz\}$ of $G$ such that $u,w\in X$, $G[X]+uw\in \mathcal{\widehat{H}}$ and $G[\overline{X}\cup\{u,w\}]+uw\in \mathcal{H}^\diamond$.
Assume that $G[\overline{X}\cup\{u,w\}]+uw=(L_n\oplus H)_{e,f}$ for some $n\geq 1$, where $H\in\mathcal{H}$ and $e$ is a 22-edge of $L_n$.
By Lemma \ref{bi-2-cut}(ii), each nice pair set $(A_1,B_1)$ (resp. $(A_H,B_H)$) of $G[X]+uw$ (resp. $H$) satisfies $|A_1|\leq3$ and $|B_1|\leq3$ (resp. $|A_H|\leq3$ and $|B_H|\leq3$).
So $G[X]+uw\cong K_{3,3}$ by Lemma \ref{bi-3-edge} and $H\in \mathcal{T}$ by the induction hypothesis.
Since $G\cong(K_{3,3}\oplus(L_n\oplus H)_{e,f})_{e'}=(H\oplus(L_n\oplus K_{3,3})_{e'})_{f,e}$, we have that $G\in \mathcal{T}$, where $e'$ is another 22-edge of $L_n$.
So the necessity holds.

Sufficiency. The result is obvious if $G=K_{3,3}$. So suppose that $G\neq K_{3,3}$. Then $|V(G)|\geq 12$.
By induction on $|V(G)|$. If $|V(G)|=12$, then $G=(K_{3,3}\oplus(L_1\oplus K_{3,3})_{e_1})_{e_2}$, where $\{e_1,e_2\}$ is a matching of $L_1$.
Since each nice pair set $(A^*,B^*)$ of $K_{3,3}$ satisfies $|A^*|\leq 3$ and $|B^*|\leq 3$, the result follows from Lemma \ref{bi-2-cut}.
Assume that $|V(G)|\geq 14$ and the result holds for graphs with less vertices than $G$.
Let $G=(T\oplus(L_{n'}\oplus K_{3,3})_{e_4,e_3})_{e_1,e_2}$ (see Fig. \ref{best-sufficiency}), where $T\in \mathcal{T}$ and $e_i=u_iv_i$ for $1\leq i\leq 4$.
Set $U=V(L_{n'})\setminus\{u_i,v_i|i=2,4\}$.
By Lemma \ref{bi-2-cut}(i), $A'\cup B'$ is contained in $V(T)$, $V(K_{3,3})$ or $U$.
Observe that for any $a\in U\cap A$ (resp. $b\in U\cap B$), there is exactly one vertex $b\in U\cap B$ (resp. $a\in U\cap A$) of $G$ such that $(a,b)$ is a nice pair of $G$.
So $|A'|=|B'|=1$ if $A'\cup B'\subseteq U$.
If $A'\cup B'\subseteq V(T)$ (or $A'\cup B'\subseteq V(K_{3,3})$), then $(A',B')$ is also a nice pair set of $T$ (or $K_{3,3}$) by Lemma \ref{bi-2-cut}(ii).
If $T\neq K_{3,3}$, then $|A_T|\leq 3$ and $|B_T|\leq 3$ for each nice pair set $(A_T,B_T)$ of $T$ by induction.
So $|A'|\leq 3$ and $|B'|\leq 3$, and the sufficiency is proved.
\begin{figure}
\centering
\includegraphics{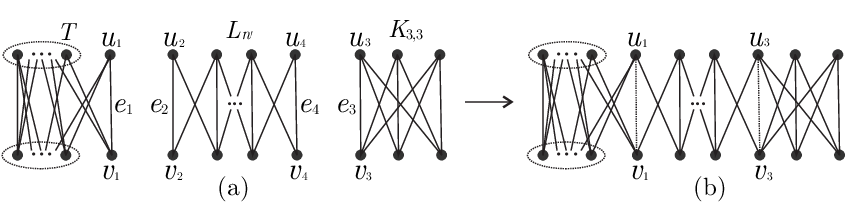}
\caption{\label{best-sufficiency} Illustration for the proof of sufficiency of Lemma \ref{best}.}
\end{figure}
\end{proof}

Theorem \ref{main-2} is obtained directly from Lemmas \ref{bi-3-edge}, \ref{2-cut-min}(ii) and \ref{best}. So we prove Corollary \ref{cor}.

\noindent{\textbf{Proof of Corollary \ref{cor}.}
If $G$ is 3-connected, then $G$ has at least 9 nice pairs of vertices and equality holds if and only if $G=K_{3,3}$ by Lemma \ref{bi-3-edge}.
If $G$ is not 3-connected, by Lemma \ref{2-cut-min}(i), take a 2-cut $\partial{(X)}=\{uv,wz\}$ of $G$ such that $u,w\in X$, $G[X]+uw\in \mathcal{\widehat{H}}$ and $G[\overline{X}\cup\{u,w\}]+uw\in \mathcal{H}^\diamond$.
Assume that $G[\overline{X}\cup \{u,w\}]+uw=(L_n\oplus H)_{e,f}$ for some $n\geq 1$, where $H\in \mathcal{H}$ and $e$ is a 22-edge of $L_n$.
Then $G[X]+uw$ and $H$ both have at least 9 nice pairs of vertices by Theorem \ref{main-2}.
Using Lemma \ref{bi-2-cut}(ii), $G$ has at least 18 nice pairs of vertices, and so the corollary holds.
\hfill$\square$

\section{\normalsize Conclusion and further problems}
In this paper, we characterized nice vertices in cubic graphs in terms of barriers (see Theorem \ref{criteria}), which was defined as trimatched vertices  in \cite{KSS}.
By using some results of matching covered graphs, we mainly obtained two lower bounds on the number of nice vertices in non-bipartite cubic graphs and determined all the corresponding extremal graphs. We also defined  nice pairs of vertices in cubic bipartite graphs and showed that a cubic bipartite graph is a brace if and only if each pair of vertices in different colors is nice (see Proposition \ref{brace-nice}). We also showed that every cubic bipartite graph has at least 9 nice pairs of vertices.

However, there are still some problems that may be further considered. Of course there is a polynomial algorithm to determine whether a vertex $u$ in a cubic graph $G$ is nice: First delete vertex $u$ together with three neighbors from $G$, then check whether the remained graph has a perfect matching or not by using known matching algorithm. Then $u$ is nice or not according to our answer ``yes" or ``no". From this we can determine all  nice vertices in cubic graphs.
\vskip 0.2cm
\noindent\textbf{Problem 1.} Is there a faster algorithm to determine all nice vertices in cubic graphs ?

\vskip 0.2cm
Lemma \ref{bicritical} states that each vertex in a cubic bicritical graph is nice.
But there  exist some non-bicritical cubic graphs in which all vertices are nice.
Let $G$ be a cubic graph obtained from a cubic brace $H(A,B)$ by splicing a cubic brick at each vertex $b\in B$. For example, see Fig. \ref{splicing}(c) for $H(A,B)=K_{3,3}$ and the cubic brick $K_4$.
Then  $G$ is a non-bicritical graph since the deletion  of any two vertices in $A$ results in a graph with no perfect matching. But each vertex of $G$ in $A$ is nice by Lemma \ref{p-3-edge-nice} and the other vertices of $G$ are also nice by Lemma \ref{G/X-nice}. So we propose the following problem.

\vskip 0.2cm
\noindent\textbf{Problem 2.} Characterize all cubic graphs whose all vertices are nice.

\vskip 0.2cm
Finally, we can also generalize the concept of nice vertices in cubic graphs to general graphs as follows.
A vertex $u$ of a graph $G$ is nice if $u$ and three of its neighbors induce a nice subgraph of $G$.
Then we can consider some similar problems  in regular graphs.
\vskip 0.2cm

\noindent\textbf{Problem 3.} How many nice vertices do $k$-regular graphs of even order have for $k\geq4$ ?

\newpage
\begin{center}
Appendix
\end{center}

Lemma 3.2 of our previous version (see [{\em Discrete Math}. 348 (2025) 114553]) states  that each vertex of a minimum nontrivial barrier of a 3-connected non-bipartite cubic graph is nice. However this statement is wrong. We give a counterexample: Let $G$ be a 3-connected non-bipartite cubic graph as shown in Fig. \ref{counterexample}(a). Then each vertex of a minimum non-trivial barrier $S$ of $G$ is not nice.
Main reason is that we (wrongly) claimed that $\overline{k}$ is nice in $G_1$ in the proof of Lemma 3.2, which is due to the incorrect statement that $K-v_1$ is contained in a component of $G'-S'$ in the third line of the proof of Claim 1.
In this version, Lemma \ref{p-3-edge-nice} can be corrected as ``Let $G$ be a 3-connected non-bicritical non-bipartite cubic graph. Then $G$ has a minimal nontrivial barrier $S$ so that each vertex of $S$ is nice.'', and Lemmas \ref{nice-number} and \ref{6-nice} remain correct and their proofs are slightly revised. Such revision of Lemma 3.2 does not affect the correctness of the other results.
\begin{figure}[ht]
\centering
\includegraphics{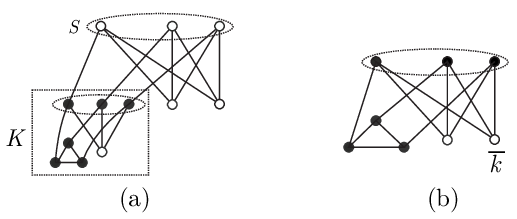}
 \caption{\label{counterexample} (a)\ A counterexample $G$ of Lemma 3.2, (b)\ $G_1$, where $\overline{k}$ is not nice in $G_1$.}
\end{figure}
\end{document}